\documentclass[10pt,twocolumn,twoside]{IEEEtran}
\usepackage{graphicx,epstopdf,multirow,multicol,booktabs,pgfplots}

\usepackage{amsmath,bm}
\usepackage{amssymb,amsthm}
\usepackage[T1]{fontenc}
\usepackage[colorlinks,citecolor=blue]{hyperref}
\allowdisplaybreaks
\usepackage{setspace}
\usepackage{algorithm}
\usepackage[noend]{algpseudocode}
\usepackage{multirow}
\usepackage[export]{adjustbox}
\usepackage{bm}

\algnewcommand\algorithmicinput{\textbf{Input:}}
\algnewcommand\Input{\item[\algorithmicinput]}

\algnewcommand\algorithmicoutput{\textbf{Output:}}
\algnewcommand\Output{\item[\algorithmicoutput]}

\usepackage{amssymb}
\usepackage{amsmath}
\usepackage{url}
\usepackage{eucal}
\usepackage{cite}
\usepackage{times}
\usepackage{array}
\newcolumntype{?}{!{\vrule width 1pt}}

\usepackage{accents}

\usepackage{hyperref}
\allowdisplaybreaks

\theoremstyle{remark}

\newtheorem{lemma}{Lemma}

\pgfplotsset{compat=1.14}

\ifCLASSOPTIONcompsoc
\usepackage[caption=false, font=normalsize, labelfont=sf, textfont=sf]{subfig}
\else
\usepackage[caption=false, font=footnotesize]{subfig}
\fi

\begin{document}
\title{\bf \Huge Credible Interdiction for Transmission Systems}

\author{Kaarthik Sundar$^{*}$, Sidhant Misra$^{\dagger}$, Russell Bent$^{\dagger}$,  Feng Pan$^a$
\thanks{$^{*}$Information Systems \& Modeling, Los Alamos National
Laboratory, USA. E-mail: \texttt{kaarthik@lanl.gov}}\;
\thanks{$^{\dagger}$Applied Mathematics and Plasma Physics Division, Los Alamos National
Laboratory, USA. E-mail: \texttt{sidhant,rbent@lanl.gov}}\;
\thanks{$^a$Optimization and Control, Pacfic Northwest National Laboratory, USA.}\;
}

\markboth{Journal of \LaTeX\ Class Files,~Vol.~14, No.~8, January~2020}%
{Sundar \MakeLowercase{\textit{et al.}}: Credible Interdiction for Transmission Systems}

\maketitle

\begin{abstract}
This paper presents novel formulations and algorithms for $N$-$k$ interdiction problems in transmission networks. In particular, it formulates spatial and topological resource constraints on attackers for $N$-$k$ interdiction problems and illustrates the formulation with two new classes of $N$-$k$ attacks: (i) Spatial $N$-$k$ attacks where the attack is constrained by geographic distance of a bus chosen by an attacker and (ii) Topological $N$-$k$ attacks where the attack is constrained to connected components. These two specific types of $N$-$k$ attacks compute interdiction plans designed to better model localized attacks, such as those induced by natural disasters or physical attacks. We then formulate these two resource-constrained interdiction problems as bilevel, max-min optimization problems and present a novel constraint generation algorithm to solve these formulations. Detailed case studies analyzing the behavior of spatially and topologically resource-constrained problems and comparing them to the traditional $N$-$k$ interdiction problem are also presented.
\end{abstract}

\begin{IEEEkeywords}
$N$-$k$ interdiction, power grids, Stackleberg game, coordinated attacks, penalty method
\end{IEEEkeywords}
\IEEEpeerreviewmaketitle

\section{Introduction} \label{sec:intro}
The electric power transmission grid plays an important and critical role in sustaining the socioeconomic systems that modern society depends on. Recent events, including natural disasters and intentional physical attacks on the grid, illustrate the need for methods that identify small sets of components whose failure leads to significant system impacts. One model that identifies such sets is the $N$-$k$ interdiction problem \cite{Salmeron2009}. The interdiction problem seeks to identify $k$ components in the system whose simultaneous or near-simultaneous failure causes the worst case disruption to the grid, where disruption is typically measured through load shed. This problem is often modeled as a bilevel, Stackelberg game (see \cite{Brown2006}) with an attacker and a defender. The attacker's and defender's actions are sequential and the attacker has a perfect model of how the defender will optimally respond to an attack. The objective of the attacker is to identify $k$ components whose loss maximizes the minimum load that a defender must shed in response to the simultaneous or near-simultaneous failure of those components. In this problem, $N$ and $k$ refer to the total number of components and the number of attacked components, respectively. The number of possible $N$-$k$ contingencies, even for small values of $k$, makes complete enumeration computationally intractable. 

One of the limitations of the $N-k$ interdiction problem is that many of the identified contingencies (especially for large $k$) are often unlikely to fail simultaneously. Indeed, a recent NERC report \cite{NercReport} discusses the value of using interdiction models to assess the resilience of transmission systems to  $N$-$k$ contingencies, however, the report stresses that the analysis should be restricted to \emph{credible} $N-k$. The report does not formally define credibility and this definition often depends on the scope of applications.  To address this gap, we develop a general approach for modeling definitions of credibility through constraints on the attacker in the interdiction problem. We demonstrate the flexibility of our approach with two models of credibility. One where the interdiction adversary is constrained spatially and one where the interdiction adversary is constrained topologically. 
These models encompass many practical situations such as coordinated physical attacks, hurricanes and earthquakes\footnote{We do not assume a natural event is intelligent. In this context, $k$ is used to model the expected number of components that fail during a natural event defined by a spatial footprint. The interdiction problem determines the components that could fail under that assumption and lead to the worst outcome}.
We then compare these worst-case-scenario models with the all powerful adversary described in the literature (referred to as the the \emph{traditional} interdiction model).
While we focus this paper on spatial and topological constraints, we stress that the techniques presented in this paper are general and are extendable to other formulations of credible interdiction in a wide variety of networked systems. Our second contribution stems from the observation that existing algorithms for solving interdiction problems are often computationally inefficient or lack optimality properties. We address this challenge by developing a novel constraint-generation algorithm that can solve these interdiction problems to optimality.

\subsection{Literature Review} \label{subsec:lit}
Over the last decade, considerable progress in modeling and developing algorithms to solve variants of the $N-k$ interdiction model in electric transmission systems and other critical infrastructure systems has been made. In this review, we focus on work developed for electric transmission systems and categorize the relevant work into three key area: (i) algorithms that 
solve the $N-k$ problem to global optimality (ii) scalable (heuristic) approaches, and (iii) attacker and defender capability models. 
Throughout this article, the term attacker is used to refer to the entity disabling portions of the power system. The term defender is used to refer to the power system operator. 

\noindent \textbf{Global Solution Methods} 
In the literature, global solution methods are primarily restricted to linear models of electric power systems, i.e. the linearized DC power flow model.
The earliest work is by Salmeron et al. \cite{Salmeron2004}, which developed a bilevel formulation of the $N-k$ problem. A Benders decomposition algorithm based on Salmeron et al.'s formulation was developed in \cite{Alvarez2004} and tested on small instances. Later, reference \cite{Motto2005} developed an approach for solving \cite{Salmeron2004}'s formulation by replacing the inner problem with its dual. Reference \cite{Arroyo2005} developed a similar approach where KKT conditions for the inner problem were developed and used to solve the problem; both these references then use a Bender's decomposition algorithm to solve the reformulated problem. In general, most global methods use Salmeron et al.'s model and rely on decomposition methods, like Bender's, to solve the problem.

The Benders decomposition approach solves the bilevel interdiction problem by dualizing the inner problem and converting the problem into a single level problem. The standard Benders decomposition algorithm \cite{Salmeron2004} is then applied to the single level problem to obtain the optimal interdiction plan. This process introduces bilinear terms which are reformulated by introducing additional auxiliary variables. 
As  a result, the technique generally cannot solve large instances.  We refer the reader to references \cite{Bao2019,Ouyang2017b} for a detailed discussion on the scalability of Benders, albeit on easier to solve general network flow interdiction problems. Other decomposition approaches include those discussed in 
\cite{Ouyang2018,Ouyang2017c}.

In summary, global methods are generally based on linear models of power flow physics, adopt decomposition approaches, and are restricted to small problems. Like these approaches, our approach uses the linearized DC model and develops a decomposition algorithm. In contrast, our decomposition approach to solve the problem to optimality and also scales to considerably larger problems.

\noindent \textbf{Scalable Solution Methods} 
In the literature, scalable solution approaches have focused on heuristic methods that do not guarantee global optimality, approaches which reduce the problem to simpler variants, or both.
The first work to develop a computationally fast approach produced a cutting-plane algorithm. \cite{Salmeron2009}. This paper tested the algorithm on a ``U.S. Regional Grid'' with 5000 buses, 5000 lines, and 500 generators.
While scalable, the cuts used by the approach were not generally valid. It is possible for the cuts to remove feasible (optimal) solutions and converge to a sub optimal solution. 
Heuristic approaches have also allowed the literature to consider more sophisticated, nonlinear models of power flows, such as the AC power flow model. See, for example, \cite{Kim2016} and \cite{Donde2008,Pinar2010} which used the AC power flow equations and approximations to these equations, respectively. Other examples of heuristic approaches include \cite{Bier2007} and \cite{Arroyo2013}.

More recently, \cite{Bienstock2010} developed computationally efficient algorithms to solve a minimum cardinality variant of the $N$-$k$ problem. Here, the objective is to find a minimum cardinality attack with a throughput less than a pre-specified bound. The approach adopts a custom cutting-plane algorithm that exploits the problem structure imposed by minimum cardinality. The paper also formulates and solves a nonlinear, continuous version of the problem where the attacker is allowed to change transmission line parameters to disrupt the system instead of simply removing transmission lines. The authors in \cite{Delgadillo2010} study the $N$-$k$ problem where the system operator is allowed to use both load shedding and line switching as defensive operations via a Benders decomposition algorithm. 

Other examples of problem simplification include contingency identification (see \cite{Enns1982, Eppstein2012, Davis2011, Soltan2016} and references therein). In particular, \cite{Eppstein2012} developed a heuristic approach to identify multiple contingencies that can initiate cascades on large transmission systems. Reference \cite{Davis2011} adopted current injection-based methods and ``line outage distribution factors'' to identify high consequence contingencies. This work and others like it use a variety of ``criticality'' measures for the different components in the system that aid in identifying these contingencies.  

In short, scalablity in the literature is achieved by focusing on heuristics or limiting features of the interdiction problem. In contrast, our approach develops one of the first, scalable global methods that does require simplification of the underlying problem beyond what is already adopted by global approaches. The mathematical foundation for our cutting-plane algorithm was initially developed by  \cite{Salmeron2009} and 
we develop the precise theoretical premise for why the original algorithm produces heuristic cuts which can remove the optimal solution. We develop the cut generation modification which allows us ensure the cuts that are generated are always valid. We also remark, that while we use the linearized DC power flow equations commonly used for global interdiction methods, our theoretical results naturally extend to convex relaxations of the AC power flow physics when strong duality holds.

\noindent \textbf{Models of Attacker and Defender Capabilities}
In the literature, to the best of our knowledge, the only paper that imposes deterministic restrictions on attacker capabilities is \cite{Ouyang2018,Ouyang2017c}. They develop a model where attacks are restricted a circular region and all components in the region are interdicted. This specific attacker model permits computing the damage in a simple manner using standard graph algorithms. We remark that the attacker's model in \cite{Ouyang2018,Ouyang2017c} differ from the $N$-$k$ interdiction model in the sense that the attacker interdicts all the components in the chosen circular region and this often leads to more that $k$ components being interdicted.
However, in general, there is limited literature where the credibility of attacks is precisely defined. The closest literature are those papers which have developed probabilistic and stochastic formulations of the $N-k$ interdiction problem, for example, \cite{Sundar2018}. The likelihood functions used in these models equate credibility with a probability model associated with simultaneous failure. 
In summary, the literature has just begun to address the open questions associated with defining credible attacks and one of our key contributions is a general framework for expressing restrictions on attach capabilities which are used to express credible interdictions.


To close out our literature review, we remark that interdiction modeling has largely ignored the effects of transients and cascading outages. The only work in the literature that we are aware of which includes transients and cascading phenomena are \cite{Wang2014,Huang2018}. In particular, \cite{Wang2014} develops an effective simulation framework to simulate cascading outages after an interdiction and integrates this information into a cutting-plane algorithm for heuristically identifying worst-case $N-k$ interdictions. Though this paper does not explicitly modeling cascades and transients, our cutting-plane algorithm supports easy integration of the simulation framework developed in \cite{Wang2014}. Finally, recent work in \cite{Sundar2019pes} also suggests that steady-state $N$-$k$ interdiction problem solutions can be used as effective surrogates to identify a minimal set of interdiction plans that can lead to cascading failures making the current work important in its own right.  

\section{Problem Formulation} \label{sec:formulations}
This section presents a bilevel, mixed-integer linear program for a general $N$-$k$ interdiction problem and then demonstrates how the formulation is specialized to model credible restrictions like spatial and topological resource restrictions.
For ease of exposition and \textcolor{black}{to retain the focus of this paper on the new contributions}, we assume that only transmission lines can be interdicted.
The formulations and algorithms can be extended to the general case where any component (i.e., generators, substations, etc.) in the transmission system can interdicted by using techniques like those discussed in \cite{Salmeron2009}.
We first present the nomenclature and terminology that we use throughout the rest of the paper. Unless otherwise stated, all the values are in per-unit (pu). \\
\noindent \emph{Sets:} \\
$\mathcal N$ - set of buses (nodes) in the network \\
$\mathcal N(i)$ - set of nodes connected to bus $i$ by an edge \\
$\mathcal E$ - set of edges (lines) in the network \\
$\mathcal E$ - set of \emph{from} edges in the network \\
$\mathcal E^r$ - set of \emph{to} edges in the network \\
$\mathcal E(i)$ - set of edges in $\mathcal E$ connected to bus $i$ and oriented \emph{from} $i$ \\
$\mathcal E^r(i)$ - set of edges in $\mathcal E$ connected to bus $i$ and oriented \emph{to} $i$ \\
\noindent \emph{Variables:} \\
$\theta_i$ - phase angle at bus $i$ \\
$p^{g}_i$ - active power generated at bus $i$\\
$\ell_i$ - percent active power (load) shed at bus $i$ \\
$p_{ij}$ - active power flow on line $(i,j)$ \\
$x_{ij}$ - binary interdiction variable for line $(i,j)$\\
$\delta_{ij}$ - auxiliary flow variable for each line $(i, j)$ \\
$y_i$ - auxiliary binary variable for bus $i \in \mathcal N$ \\
$\bm x$ - vector of interdiction variables $x_{ij}$ \\
$x_i$ - binary variable for bus $i \in \mathcal N$, indicating if the bus is the center of the circular footprint for the spatial $N$-$k$ interdiction problem\\
\noindent \emph{Constants:} \\
$ \bm p^d_i$ - active power demand at bus $i$ \\
$ \bm X_{ij}$ - magnitude of reactance of line $(i, j)$ \\
$\bm b_{ij}$ - for line $(i, j)$ defined by $1/{\bm X_{ij}}$ \\
$(0,  \overline{\bm p}_i)$ - bounds for active power generated at bus $i$ \\
$\bm t_{ij}$ - thermal limit of line $(i,j)$ \\
$k$ - number of interdicted components \\
$D$ - planar distance limit \\
\noindent For a value $(\cdot)$, we use the notation $(\cdot)^+$ $(\cdot)^-$ as shorthand for $\max\{0,(\cdot)\}$ and $-\min\{0, (\cdot)\}$, respectively
Using the above set of notations, the $N$-$k$ interdiction problem is formulated as follows:
\begin{flalign}
(\mathcal F) \quad  \max_{\bm x \in X} \; \eta\, (\boldsymbol x) \label{eq:outer-obj}
\end{flalign}
where, 
\begin{subequations}
\begin{flalign}
& (\text{LS}(\boldsymbol x)) \quad \eta(\boldsymbol x) = \min \sum_{i \in \mathcal N}  \bm p_i^d \ell_i \quad \text{subject to:} \label{eq:inner-obj} & \\
& p_i^g - (1-\ell_i) \bm p_i^d = \sum_{(i,j) \in \mathcal E(i)} p_{ij} - \sum_{(j, i) \in \mathcal E^r(i)} p_{ji} \quad \forall i \in \mathcal N, & \label{eq:kcl} \\
& 0 \leqslant p^{g}_i \leqslant \overline{\bm p}_i  \quad \forall i \in \mathcal N, &\label{eq:pg} \\
& - \bm M x_{ij} \leqslant p_{ij} +  \bm b_{ij} (\theta_i - \theta_j) \leqslant \bm M x_{ij} \quad \forall (i,j) \in \mathcal E \label{eq:pij} & \\
&  -\bm  t_{ij}(1-x_{ij}) \leqslant p_{ij} \leqslant \bm  t_{ij}(1-x_{ij}) \quad \forall (i,j) \in \mathcal E \label{eq:thermal} & \\
&0 \leqslant \ell_i \leqslant 1 \quad \forall i \in \mathcal N. & \label{eq:loadshed}
\end{flalign}
\label{eq:inner}
\end{subequations}
\noindent and $\bm M = \sum_{i \in \mathcal N}  \bm p_i^d$.
For the traditional $N$-$k$ interdiction problem, the set $X$ is given by $\{\bm x :\sum_{(i,j) \in \mathcal E} x_{ij} = k \}$. Additional restrictions on the set $X$ enforce resource
restrictions on the $N$-$k$ interdiction plans. The inner problem is a linear program for a fixed interdiction plan. It formulates a minimum load shedding problem using the linearized DC power flow constraints. 
Eq. \eqref{eq:inner-obj} describes the load shedding objective of the defender. Constraints \eqref{eq:kcl} enforces flow balance at each node (Kirchoff's Law). Constraints \eqref{eq:pg} enforce limits on generators and constraints \eqref{eq:pij} relates flow on a line to the voltage phase angle difference across the line (Ohm's Law). Constraints \eqref{eq:thermal} enforces thermal limits on each power line.
 Constraints \eqref{eq:pij} and \eqref{eq:thermal} are functions of the interdiction variables $\bm x$
which activate or deactivate a line.
We remark that this model does not shed load that is co-located with a generator because the model only interdicts lines. To achieve this effect, buses can be split into separate generator and load buses with no geographic distance. Alternatively, bus interdiction variables can be included in the formulation to shed loads that are co-located with generators in the same bus (see \cite{Salmeron2009} for details). We next discuss a spatial and a topological example of modeling restrictions on $X$. 

\subsection{Spatial $N$-$k$ Interdiction} \label{subsec:spatial}
This model forces the choice of $k$ lines to fall within a geographic footprint.  Instances of this kind of an interdiction plan include natural disasters which are localized, localized cyber attacks, and physical attacks. To illustrate spatially constrained interdiction, in this formulation the attacker chooses
a bus $i$ such that the $k$ lines are within a planar distance of $\frac D2$ units from the bus $i$. This enforces the interdicted lines to be within a planar distance of $D$ units of one another. To formulate the problem, we use $\varphi_{ij}$ to denote the set of all buses within a planar distance, $\frac D2$, of line $(i,j)$.
Given $\varphi_{ij}$ for each line $(i,j)$, the set $X$ for spatial $N$-$k$ interdiction is modeled as:
\begin{subequations}
\begin{flalign}
&\sum_{(i,j) \in \mathcal E} x_{ij} \leqslant k, \,\, \sum_{i \in \mathcal N} x_i = 1 \label{eq:spatial_k}  \\
&x_{ij} \leqslant \sum_{n \in \varphi_{ij}} x_n \quad \forall (i,j) \in \mathcal E, & \label{eq:spatial_d} \\
& x_{ij} \in \{0, 1\} \; \forall (i,j) \in \mathcal E, \; x_i \in \{0,1\} \; \forall i \in \mathcal N. & 
\end{flalign}
\label{eq:spatial}
\end{subequations}
Constraints \eqref{eq:spatial_k} ensure (i) at most $k$ components are interdicted and (ii) exactly one bus is selected. Unlike the traditional interdiction, the spatial interdiction problem may not contain a feasible interdiction plan with exactly $k$ lines removed. This may occur when $D$ is very small and there are less than $k$ components of the transmission system in the region of impact. Without loss of generality, if the location of the event's bus is known, then the binary variable for that bus can explicitly be set to one. Finally, constraints \eqref{eq:spatial_d} ensure that all the interdicted components lie within the circle of radius $\frac D2$ units centered around the chosen bus. 
Fig. \ref{fig:illustration} illustrates the condition modeled by constraints \eqref{eq:spatial}. Here the impact is centered around the bus $5$. Hence, the value of $x_5$ is set to $1$. Then, constraint \eqref{eq:spatial_k} enforces all the other binary variables that correspond to each bus to be $0$ (as illustrated in the figure). Constraints \eqref{eq:spatial_d} ensure that the transmission lines whose center is within a distance of $\frac D2$ units from the bus $5$ (the impact region) are the only lines that can be interdicted. In Fig. \ref{fig:illustration}, they are the lines colored in red. In this set-up, the attacker can now choose at most $k$ lines from these lines that are within the circular impact region.

\begin{figure}
    \centering
    \includegraphics[scale=0.8]{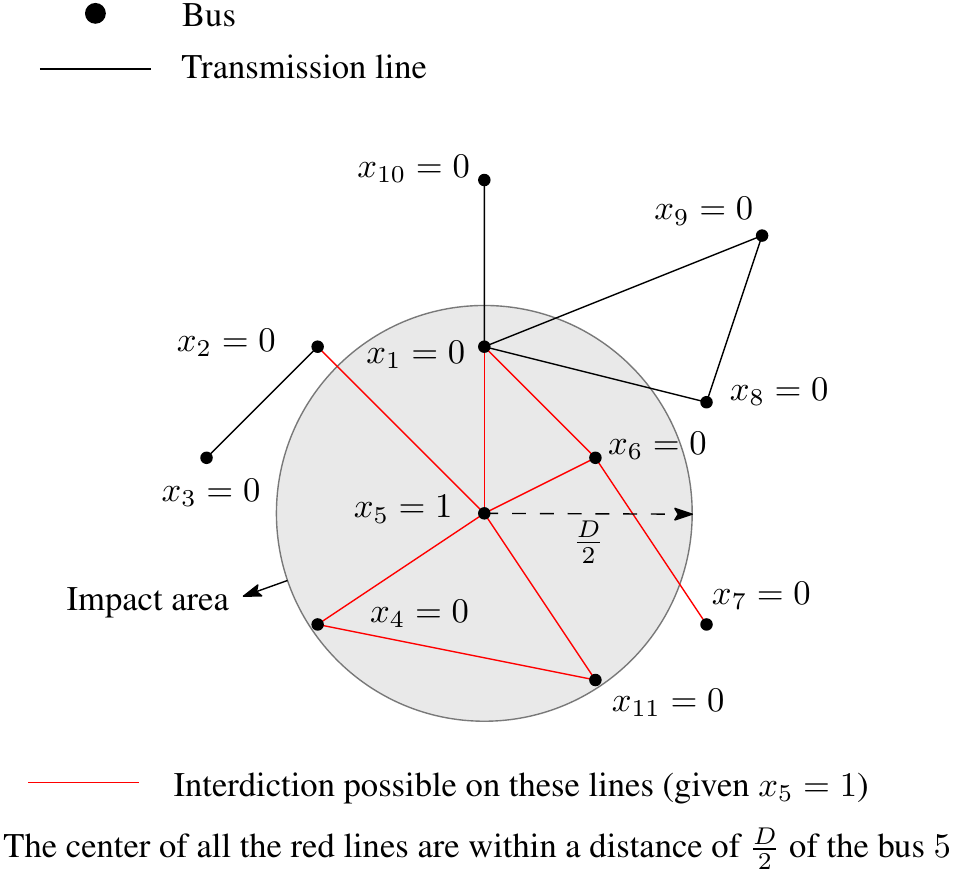}
    \caption{Illustration of the condition enforced by constraints \eqref{eq:spatial}.}
    \label{fig:illustration}
\end{figure}

\subsection{Topological $N$-$k$ Interdiction} \label{subsec:topological} 
The spatial $N$-$k$ interdiction limits the damaging event's region of impact to a geographic region and determines the worst-case $N$-$k$ attack within this footprint. But damaging events like tornadoes, hurricanes, etc. may not be circular or even regularly shaped. Furthermore, spatial data might not be available in many cases. Hence, as a proxy, the topological $N$-$k$ interdiction considers connected components, i.e., it computes a worst-case connected $N$-$k$ attack. To formulate the connectivity constraints, we use a single-commodity flow formulation \cite{Magnanti1995} by introducing a ``super sink'' bus ($\kappa$) that is connected to every bus in the network. 
This flow has no relationship to electrical flow. Rather, it is a virtual formulation that forces the attacker to choose connected components.
$\kappa$ serves as a source of $k+1$ units of virtual flow that are carried to all the buses in the network that are incident on at least one interdicted line. Flow variables, $\delta_{ij}$, for each $(i, j) \in \mathcal E$ determine the amount of flow carried on the interdicted edges. The set $X$ for the topological $N$-$k$ interdiction is modeled as follows:
\begin{subequations}
\begin{flalign}
& \sum_{(i,j) \in \mathcal E} x_{ij} = k, \,\, \sum_{i \in \mathcal N} x_{\kappa i} = 1,  \label{eq:top_k}  \\ 
& x_{ij} \leqslant y_i, \; x_{ij} \leqslant y_j \quad \forall (i,j) \in \mathcal E, \label{eq:top_aux} \\ 
& \delta_{\kappa i} \leqslant (k+1) \cdot x_{\kappa i} \quad \forall i \in \mathcal N, \label{eq:top_dummy_bd} \\
& \delta_{ij} \leqslant k \cdot x_{ij}, \,\, \delta_{ji} \leqslant k \cdot x_{ij} \quad \forall (i,j) \in \mathcal E, \label{eq:top_bd} \\ 
& \sum_{j \in \mathcal N(i)} \left(\delta_{ji} - \delta_{ij}\right) = y_i - \delta_{\kappa i} \quad \forall i \in \mathcal N, \label{eq:top_flow} \\  
& x_{ij} \in \{0, 1\} \quad \forall (i,j) \in \mathcal E, \\ 
& \delta_{ij} \geqslant 0 \quad \forall (i,j) \in \mathcal E\cup \mathcal E^r, \label{eq:deltadef1} \\ 
& \delta_{\kappa i} \geqslant 0 \quad \forall i \in \mathcal N, \\
& y_i \in \{0,1\} \; \forall i \in \mathcal N. & 
\end{flalign}
\label{eq:top}
\end{subequations}
Constraint \eqref{eq:top_k} ensures that the number of interdicted lines is $k$ and that a dummy edge originating from $\kappa$ is chosen to ensure flow is delivered to the buses incident on interdicted lines. Constraints \eqref{eq:top_aux} enforce the auxiliary binary variables $y_i$ and $y_j$ to take value $1$ if the line $(i,j)$ is interdicted. Constraints \eqref{eq:top_dummy_bd}--\eqref{eq:top_flow} are single-commodity flow constraints which ensure that the interdicted lines form a connected component. Intuitively, the single commodity flow formulation ensures that each bus incident on an interdicted line receives one unit of flow from $\kappa$. Figs. \ref{fig:illustration-top1} and \ref{fig:illustration-top2} illustrate the flow variables and the values they take for an $N$-$3$ attack. As indicated in Fig. \ref{fig:illustration-top1}, each transmission line $(i,j)$ in the original system has two non-negative flow variables $\delta_{ij}$ and $\delta_{ji}$ (see Eq. \eqref{eq:deltadef1}). The $\delta_{\kappa i}$ variable for each $i \in \mathcal N$ are also shown in Fig. \ref{fig:illustration-top1}. Given a value of $k$, the flow variables $\delta_{\kappa i}$ can carry at most $(k+1)$ units of flow depending on which dummy edge $x_{\kappa i}$ is selected by the formulation. Intuitively, the single-commodity flow constraints serve the purpose of ensuring that the interdicted lines form a connected component using the flow graph shown in Fig. \ref{fig:illustration-top1}. Fig. \ref{fig:illustration-top2} shows a feasible $N$-$3$ attack with $3$ interdicted lines and corresponding feasible flows. In this case, $4$ units of flow originate from the super-sink, $\kappa$, and one unit of flow is consumed by each bus in the set $\{1, 2, 3, 5\}$. The flow values satisfy the single-commodity flow constraints in \eqref{eq:top_dummy_bd}--\eqref{eq:top_flow}. 

\begin{figure}
    \centering
    \includegraphics[scale=0.8]{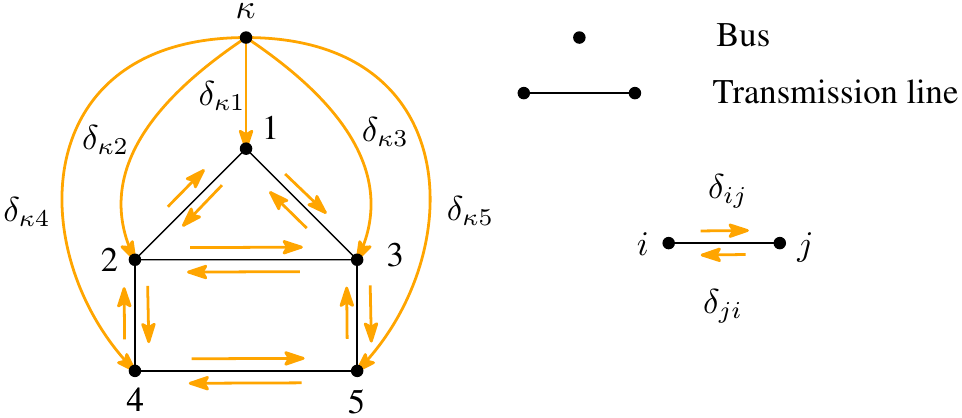}
    \caption{The flow variables that correspond to a given network and the super-sink bus $\kappa$ are shown.}
    \label{fig:illustration-top1}
\end{figure}

\begin{figure}
    \centering
    \includegraphics[scale=0.7]{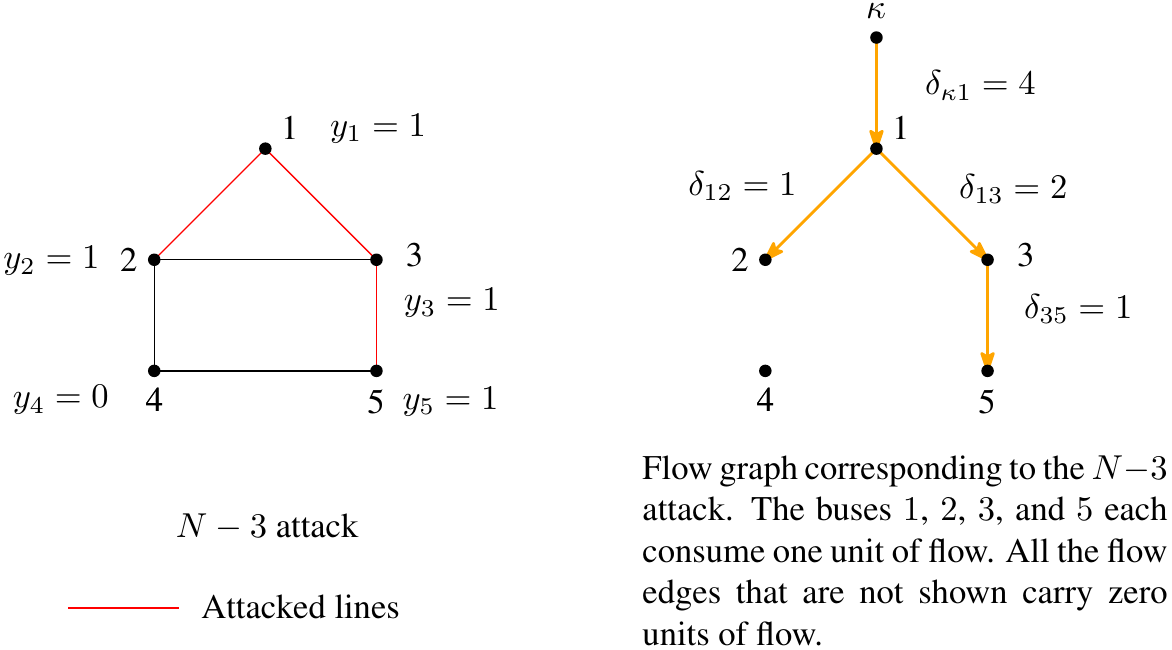}
    \caption{In this example, the $N$-$3$ attack is given by $x_{12}=1$, $x_{13}=1$, and $x_{35}=1$.}
    \label{fig:illustration-top2}
\end{figure}

Other credible resource constraints can be modeled by further adopting the framework on imposing constraints on the set $X$. In the next section, we develop a penalty-based constraint-generation algorithm to solve the bilevel problems to optimality. 

\section{Solution Methodology} \label{eq:algorithm} 
This section describes a novel constraint-generation algorithm that works directly with the bilevel structure of the interdiction problem 
and does not dualize the inner problem. We begin by first presenting a reformulation of the inner problem LS$(\bm x)$. 
\subsection{Reformulation of the interdiction problem}
The reformulation depends on the observation that LS$(\bm x)$ is feasible for all $\bm x \in X$. In this case, the operator can shed all load in the system to make the inner problem feasible, so this restriction is not artificial. The reformulation depends on the following property of linear programs.
\begin{lemma}
\label{lem:dual_bound}
Consider a linear program of the following form 
\begin{flalign}
\text{(P1) :} \quad \min \{ c \cdot z : A z \geqslant  b, z \geqslant 0 \} \label{eq:p1}
\end{flalign}
that has a finite optimum. Let $\alpha$ denote the vector of dual variables that correspond to the constraints $A z \geqslant b$. Let $\bar{\alpha}$ be a vector of upper bounds on the optimal dual variables ${\alpha}^*$. Consider the following optimization problem:
\begin{flalign}
\text{(P2) : } \quad \min \{ c \cdot z + \bar{\alpha} \cdot (b-Az)^+ : z \geqslant 0\} \label{eq:p2}
\end{flalign}
where, $(b-Az)^+$ is the positive part of $(b-Az)$, i.e., the vector of constraint violations of $Az \geqslant b$.
Then, the optimal values of (P1) and (P2) are identical.
\end{lemma}
\begin{proof}
Let $y^*$ denote the optimal solution value of the LP in (P1). The dual of (P1) is 
\begin{flalign}
\text{(D1) : } \quad y^* = \max \{ b \cdot \alpha : \alpha A \leqslant c, \alpha \geqslant 0 \}. \label{eq:dual}
\end{flalign}
Strong duality for LPs result in the optimal solution value of (D1) being equal to $y^*$. From the statement of the Lemma, we add the constraint $\alpha \leqslant \bar{\alpha}$ without changing the optimal solution value to (D1) to obtain 
\begin{flalign}
\text{(D1a) : } \quad y^* = \max \{ b \cdot \alpha : \alpha A \leqslant c, \alpha \leqslant \bar{\alpha}, \alpha \geqslant 0 \}. \label{eq:dual-1}
\end{flalign}
We now observe that the dual of the LP in (D1a) yields (P2) which results in the optimal solution value of (P2) being equal to $y^*$.  
\end{proof}
We remark that though the above Lemma is presented for any LP in its canonical form, it can be applied to any arbitrary LP after converting it to the form given in Eq. \eqref{eq:p1}. 

For a fixed interdiction plan $\bm x \in X$, we let $(\mu_{ij}^1(\bm x), \mu_{ij}^2(\bm x))$ and $(\pi_{ij}^1(\bm x),\pi_{ij}^2(\bm x))$  denote the optimal dual variables for constraints \eqref{eq:pij} and \eqref{eq:thermal} (after converting them to the canonical form given by Eq. \eqref{eq:p1}), respectively.  The notation $(\bar \mu_{ij}^1(\bm x), \bar \mu_{ij}^2(\bm x))$ and $(\bar \pi_{ij}^1(\bm x),\bar \pi_{ij}^2(\bm x))$ denotes the upper bounds of the dual variables over each $\bm x \in X$. For the sake of clarity, the constraints in their canonical forms, the dual variables and the corresponding upper bounds on the optimal dual variables are shown below.
\begin{subequations}
\begin{flalign}
p_{ij} + \bm b_{ij} (\theta_i - \theta_j)\geqslant - \bm M x_{ij} ~~ &\rightarrow ~~ (\mu_{ij}^1(\bm x), \bar \mu_{ij}^1(\bm x)) \label{eq:pij-d1} \\
-p_{ij} - \bm b_{ij} (\theta_i - \theta_j) \geqslant - \bm M x_{ij} ~~ &\rightarrow ~~ (\mu_{ij}^2(\bm x), \bar \mu_{ij}^2(\bm x) \label{eq:pij-d2} \\
p_{ij} \geqslant -\bm  t_{ij}(1-x_{ij}) ~~ &\rightarrow ~~ (\pi_{ij}^1(\bm x), \bar \pi_{ij}^1(\bm x)) \label{eq:thermal-d1} \\ 
-p_{ij} \geqslant - \bm t_{ij}(1-x_{ij}) ~~ &\rightarrow ~~ (\pi_{ij}^2(\bm x), \bar \pi_{ij}^2(\bm x)) \label{eq:thermal-d2} 
\end{flalign}
\label{eq:dual-bounds}
\end{subequations}
Notice that in Eq. \eqref{eq:dual-bounds}, the two-sided constraints in \eqref{eq:pij} and \eqref{eq:thermal} are equivalently represented using two one-sided constraints in canonical form. Doing so makes all the dual variables corresponding to those constraints non-negative and, hence, the upper bounds on the optimal dual variables for any value of $\bm x$ are non-negative. Now, we apply Lemma~\ref{lem:dual_bound} to the inner problem to obtain
\begin{subequations}
\begin{flalign}
& \begin{gathered}
 \eta^r(\boldsymbol x) = \min \sum_{i \in \mathcal N}  p_i^d \ell_i \, + \\
\sum_{(i,j) \in \mathcal E} \left\{ 
\begin{gathered}
 \bar \mu_{ij}^1(\bm x) [p_{ij} + \bm b_{ij} (\theta_i - \theta_j) + \bm Mx_{ij}]^-   \, +  \\ 
 \bar \mu_{ij}^2 (\bm x)[ p_{ij} + \bm b_{ij} (\theta_i - \theta_j) - \bm Mx_{ij}]^+  \, + \\ 
 \bar \pi_{ij}^1 (\bm x) [p_{ij} + \bm t_{ij}(1-x_{ij})]^- \,+\\
  \bar\pi_{ij}^2 (\bm x) [p_{ij} - \bm t_{ij}(1-x_{ij})]^+   
\end{gathered} 
\right\}
\end{gathered} \label{eq:inner-obj-r} & \\ 
& \text{subject to: Eqs \eqref{eq:kcl}, \eqref{eq:pg}, \eqref{eq:loadshed}, }& \notag \\ 
& - \bm M \leqslant p_{ij} +  \bm b_{ij} (\theta_i - \theta_j) \leqslant \bm M  \quad \forall (i,j) \in \mathcal E, \label{eq:pij-r} & \\
&  -\bm t_{ij} \leqslant p_{ij} \leqslant  \bm t_{ij}\quad \forall (i,j) \in \mathcal E. \label{eq:thermal-r} &
\end{flalign}
\label{eq:reformulation}
\end{subequations}
\textcolor{black}{In Eq. \eqref{eq:inner-obj-r}, $[z]^+$ and $[z]^-$ are the positive and negative parts of $z$, i.e., $[z]^+ = \max(z, 0)$ and $[z]^- = \max(-z, 0)$.} This reformulation in Eq. \eqref{eq:reformulation} is a variant of ``penalty methods'' used to solve non-linear optimization problems \cite{Bertsekas1999}. 
Since $x_{ij} \in \{0,1\}$, the objective in \eqref{eq:inner-obj-r} is equivalent to 
\begin{flalign}
& \begin{gathered}
 \eta^r(\boldsymbol x) = \min \sum_{i \in \mathcal N}  p_i^d \ell_i \, + \\
\sum_{(i,j) \in \mathcal E} \left\{ 
\begin{gathered}
 \bar \mu_{ij}^1(\bm x) [p_{ij} + b_{ij} (\theta_i - \theta_j)]^-  (1 - x_{ij}) \, +  \\ 
 \bar \mu_{ij}^2 (\bm x)[p_{ij} + b_{ij} (\theta_i - \theta_j)]^+  (1 - x_{ij}) \, + \\ 
 \bar \pi_{ij}^1 (\bm x) p_{ij}^-  x_{ij} +  \bar \pi_{ij}^2 (\bm x) p_{ij}^+  x_{ij} 
\end{gathered} 
\right\}
\end{gathered}  \label{eq:objective_reformulated}
\end{flalign}
\textcolor{black}{We detail the explanation for the equivalence between Eqs. \eqref{eq:inner-obj-r} and \eqref{eq:objective_reformulated} for the first term in Eq. \eqref{eq:inner-obj-r}, i.e., $\bar \mu_{ij}^1(\bm x) [p_{ij} + \bm b_{ij} (\theta_i - \theta_j) + \bm Mx_{ij}]^-$. When $x_{ij} = 0$ and $1$, $\bar \mu_{ij}^1(\bm x) [p_{ij} + \bm b_{ij} (\theta_i - \theta_j) + \bm Mx_{ij}]^-$ reduces to $\bar \mu_{ij}^1(\bm x) [p_{ij} + \bm b_{ij} (\theta_i - \theta_j)]^-$ and $\bar \mu_{ij}^1(\bm x) [p_{ij} + \bm b_{ij} (\theta_i - \theta_j) + \bm M]^-$, respectively. But we observe from Eq. \eqref{eq:pij-r} that, $p_{ij} + \bm b_{ij} (\theta_i - \theta_j) + \bm M$ is always $\geqslant 0$, i.e., $\bar \mu_{ij}^1(\bm x) [p_{ij} + \bm b_{ij} (\theta_i - \theta_j) + \bm M]^- = 0$, proving the equivalence. A similar argument holds for the rest of the terms in Eqs. \eqref{eq:inner-obj-r} and \eqref{eq:objective_reformulated}.}

Our proposed constraint generation approach is based on estimating the dual upper bounds $\bar \mu_{ij}^1(\bm x),  \bar \mu_{ij}^2(\bm x), \bar \pi_{ij}^1(\bm x) , \bar \pi_{ij}^2 (\bm x)$. A similar approach was used in \cite{cormican1998stochastic} where the authors used a \emph{constant} upper bound of $1$ for every dual variable; their approach is valid for the special case they consider where the inner problem is a standard network flow problem.
In this paper, we propose to use a piece-wise constant upper bound on one pair of dual variables, $\mu_{ij}^1(\bm x),  \mu_{ij}^2(\bm x)$ that is based on the observation in the forthcoming Lemma \ref{lem:zero_contribution}. For the other pair of dual variables $\pi_{ij}^1(\bm x) , \pi_{ij}^2 (\bm x)$, we propose using constant (both valid and heuristic) upper bounds that are computed using the load on the system. 
\begin{lemma}   \label{lem:zero_contribution}
For any interdiction plan $\bm x \in X$ where $X$  models a traditional, topological or spatial interdiction problem, the optimal dual values, $\mu_{ij}^1(\bm x)$ and $\mu_{ij}^2(\bm x)$ satisfy
\begin{align}   \label{eq:zeros}
    \mu_{ij}^1(\bm x) = \mu_{ij}^2(\bm x) = 0, \quad \mbox{whenever}\quad  x_{ij} = 1.
\end{align}
\end{lemma}
\begin{proof}
The proof follows by applying the complementary slackness condition to the constraint in Eq. \eqref{eq:pij-d1} and \eqref{eq:pij-d2}. Whenever $x_{ij} = 1$, the constraints reduce to 
\begin{subequations}
\begin{flalign}
p_{ij} + \bm b_{ij} (\theta_i - \theta_j)\geqslant - \bm M  \label{eq:pij-d1-x} \\
-p_{ij} - \bm b_{ij} (\theta_i - \theta_j) \geqslant - \bm M \label{eq:pij-d2-x}
\end{flalign}
\label{eq:pij-duals-x}
\end{subequations}
Complementary slackness conditions, when applied to constraint in Eq. \eqref{eq:pij-d1-x}, state that when the constraint is not tight, i.e., when $p_{ij} + \bm b_{ij} (\theta_i - \theta_j) > - \bm M$, the corresponding dual variable's value is zero ($\mu_{ij}^1(\bm x) = 0$). A similar statement is also true for the constraint in Eq. \eqref{eq:pij-d2-x}. Now, we observe that the value of $\bm M$ in Eq. \eqref{eq:pij} can always be made sufficiently large so that the constraints in Eq. \eqref{eq:pij-duals-x} are never tight. This observation is clear from the fact that the maximum and minimum value of $p_{ij} + \bm b_{ij} (\theta_i - \theta_j)$ when $x_{ij}=1$ is finite. Therefore, the corresponding dual variables values must always be zero for that chosen value of $\bm M$.
\end{proof}

Exploiting the property in Lemma~\ \ref{lem:zero_contribution},  the objective in  Eq. \eqref{eq:objective_reformulated} now reduces to 
$$\eta^r(\boldsymbol x) = \min \sum_{i \in \mathcal N}  \bm p_i^d \ell_i \, + \sum_{(i,j) \in \mathcal E} (\bar \pi_{ij}^1(\bm x)  p_{ij}^-   +  \bar \pi_{ij}^2(\bm x)  p_{ij}^+  ) \cdot x_{ij} $$
Choosing constant dual upper bound values of $(\bar{\pi}_{ij}^1, \bar{\pi}_{ij}^2)$ for $(\bar{\pi}_{ij}^1(\bm x), \bar{\pi}_{ij}^2(\bm x))$, respectively, we obtain the following equivalent formulation of the full interdiction problem:
\begin{subequations}
\begin{flalign}
& (\mathcal F_1) \quad  \max_{\bm x \in X} \; \eta \quad \text{subject to:} & \\ 
& \eta \leqslant \sum_{i \in \mathcal N} \bm p_i^d \ell_i(\bm x) + \sum_{(i,j) \in \mathcal E} (\bar \pi_{ij}^1  p_{ij}^-(\bm x)   +  \bar \pi_{ij}^2  p_{ij}^+(\bm x)  )  \cdot x_{ij} \label{eq:constraints} & 
\end{flalign}
\label{eq:new-model}
\end{subequations}
where, $\ell_i(\bm x)$, $p_{ij}^+(\bm x)$, and $p_{ij}^-(\bm x)$ are the optimal load shed values, $i\in \mathcal N$, and the positive and negative parts of the active power flow on each transmission line of the network for a fixed $N$-$k$ attack defined by $\bm x$. 

\subsection{Constraint generation algorithm}
The formulation in Eq. \eqref{eq:new-model} is solved using a constraint generation algorithm that alternates between solving the outer maximization problem in Eq. \eqref{eq:outer-obj} and the inner problem LS$(\bm x)$. Constraints of the form \eqref{eq:constraints} are generated at each iteration. It first relaxes all the constraints in Eq. \eqref{eq:constraints} from \eqref{eq:new-model} and solves the resulting problem to obtain an upper bound to the optimal objective value of the interdiction problem. The solution is then used to solve the inner problem i.e., a minimum load shedding problem with the components in the current solution removed. The optimal load shedding factors and active power flow on each line are then used to add a cut of the form \eqref{eq:constraints}.
The constraint-generation algorithm generates the constraints in \eqref{eq:constraints} dynamically until an optimality tolerance is reached. At each iteration of the constraint-generation algorithm, in addition to the constraint in step \ref{step:master} of Algorithm \ref{algo:pseudocode}, we also add the following no-good cut in Eq. \eqref{eq:no-good}. These no-good cuts eliminate the selection of the same interdiction plan in the outer problem after it is selected once. The cuts have been observed to improve the convergence behavior of the algorithm for the traditional interdiction problems \cite{Salmeron2009}.   
\begin{flalign}
\sum_{(i,j) \in \mathcal E} \hat{x}_{ij} \cdot x_{ij} \leqslant k - 1. \label{eq:no-good}
\end{flalign}
For clarity, pseudo-code of the algorithm is given in Algorithm \ref{algo:pseudocode}.
\begin{algorithm}
\onehalfspacing
\caption{Constraint-generation algorithm}\label{algo:pseudocode}
\begin{algorithmic}[1]
\vspace{1ex}
\Input optimality tolerance, $\varepsilon > 0$
\Output $\bm x^* \in X$, an $\varepsilon$-optimal $N$-$k$ attack
\State initial problem: $\mathcal F_{1}$ without constraint \eqref{eq:constraints}
\State $\eta^{*} \gets -\infty$ \Comment{lower bound on the optimal obj. value}
\State $\eta^{u} \gets +\infty$ \Comment{upper bound on the optimal obj. value} 
\State $\hat{\bm x}$ any initial $N$-$k$ attack
\State \label{step:inner}solve inner problem for $\hat{\bm x}$ 
\State $\eta(\hat{\bm x}) \gets$  the load shed for the attack defined by $\hat{\bm x}$
\State $\ell_i(\hat{\bm x}) \gets $ active load shed factor for $\hat{\bm x}$ and $\forall i \in \mathcal N$
\State $p_{ij}(\hat{\bm x}) \gets$ active power through line $(i,j) \in \mathcal E$
\If{$\eta(\hat{\bm x}) > \eta^*$} $\eta^* \gets \eta(\hat{\bm x})$ and ${\bm x}^* \gets \hat{\bm x}$ 
\EndIf
\State \label{step:master}add $\eta \leqslant \eta(\hat{\bm x}) + \sum_{(i,j) \in \mathcal E}  (\bar \pi_{ij}^1  p_{ij}^+(\hat{\bm x})   +  \bar \pi_{ij}^2  p_{ij}^-(\hat{\bm x})  )\cdot x_{ij}$ to $\mathcal F_{1}$ and resolve the outer problem
\State update $\hat{\bm x}$, and set $\eta^u$ using solution from Step \ref{step:master}
\If{$\eta^u - \eta^* \leqslant \varepsilon \eta^*$} $(\bm x^*, \eta^*)$ is the $\varepsilon$-optimal solution to the interdiction problem, stop  
\EndIf
\State Add no-good cut \eqref{eq:no-good}
\State return to step: \ref{step:inner}
\end{algorithmic}
\end{algorithm}

\subsection{Choosing bounds on the dual variables} \label{subsec:dual-bounds}
We also note that tight upper bounds on the optimal dual values over all $N$-$k$ attacks i.e., $\bar{\pi}_{ij}^1$ and $\bar \pi_{ij}^2$ for each line $(i,j) \in \mathcal E$ are essential to obtain reasonable convergence behavior. A trivial and valid value for the dual upper bounds for any line $(i,j) \in \mathcal E$ is the total load in the system i.e., $\bar{\pi}_{ij}^1 = \bar{\pi}_{ij}^2 = \sum_{i \in \mathcal N} \bm p_i^d$. Intuitively, the tightest value of the dual upper bounds specify the minimum amount of load that can be served by increasing the thermal limit of the corresponding line by one unit over all $N$-$k$ attacks. Given this interpretation, the total load is a very conservative choice for $\bar{\pi}_{ij}^1$ and $\bar{\pi}_{ij}^2$. In the subsequent paragraphs, we focus on (i) finding lines for which $\bar{\pi}_{ij}^1 = 0$ or $\bar{\pi}_{ij}^2 = 0$ and (ii) heuristic choices of dual upper bounds. 

Let $\bar{\mathcal E}$ denote the lines $(i,j) \in \mathcal E$ in the transmission network that are never congested (i.e., constraint \eqref{eq:thermal-r} is never tight) for any feasible interdiction plan $\bm x \in X$. Then, for any $(i,j) \in \bar{\mathcal E}$, we know that $\bar{\pi}_{ij}^1 = \bar{\pi}_{ij}^2 = 0$. The lines for which $p_{ij} < \bm t_{ij}$ for any $\bm x \in X$ have $\bar{\pi}_{ij}^1 = 0$ and lines that always satisfy $p_{ij} > -\bm t_{ij}$ for any $\bm x \in X$, $\bar{\pi}_{ij}^2 = 0$. It remains to compute the lines that satisfy these conditions for any interdiction plan. To that end, for each line, we solve the following pair of linear programs which aid in finding the lines that are never congested for any feasible $\bm x \in X$:
\begin{subequations}
\begin{flalign}
& \max \pm p_{ij} \quad \text{subject to:} & \label{eq:lp-obj} \\ 
& \text{Eq. \eqref{eq:kcl} -- \eqref{eq:loadshed}, $\bm x \in X$ and $0 \leqslant x_{ij} \leqslant 1$.} & \notag 
\end{flalign}
\label{eq:lp}
\end{subequations}
If the maximum value of $p_{ij}$ ($-p_{ij}$) is less than $t_{ij}$, then $\bar{\pi}_{ij}^1$ ($\bar{\pi}^2_{ij}$) takes a value zero i.e., the line is never congested in any interdiction plan $\bm x \in X$. For the remaining lines, we set $\bar{\pi}_{ij}^1 = \bar{\pi}_{ij}^2 = \sum_{i \in \mathcal N} \bm p_i^d$. This choice of dual upper bounds can still be weak for systems with a large total load value. Hence, we also examine heuristic choices of the dual upper bounds. In particular, we use $\bar{\pi}_{ij}^1 = \bar{\pi}_{ij}^2 = 1$ for all lines $(i,j) \in \mathcal E$. This is a valid bound for the transportation model, i.e., the inner problem LS$(\bm x)$ without the constraints \eqref{eq:pij}. Intuitively, these upper bounds have the following interpretation: if a line $(i,j) \in \mathcal E$ is removed from a network, then the maximum amount of load that is shed is the absolute value of the power flowing on the line. This is not always true because of Braess' paradox \cite{Salmeron2009}. Nevertheless, if the total load shed induced by the removal of any subset of lines $\hat{\mathcal E}$ in the network does not incur more than $\sum_{(i,j) \in \mathcal E} |p_{ij}|$ load shed, then $\bar{\pi}_{ij}^1 = \bar{\pi}_{ij}^2 = 1$ is a valid set of upper bounds for the optimal dual variables over all feasible interdiction plans $\bm x \in X$. Indeed, we see in our experiments in the next section that the choice $\bar{\pi}_{ij}^1 = \bar{\pi}_{ij}^2 = 1$ is valid for the test cases we consider.

\section{Case Studies} \label{sec:results}
In this section, we present two case studies to demonstrate the computational effectiveness of the constraint-generation algorithm in computing an optimal $N$-$k$ interdiction plan. The case studies are performed on two PGLib-OPF v18.08 \cite{pglib} API test cases: the IEEE single-area RTS96 with 24 buses and the geolocated WECC 240 with 240 buses. The IEEE single-area RTS96 \cite{Grigg1999} test system is artificially geolocated in the state of Utah. The $k$ values for both the spatial and the topological interdiction problems are varied from 2 to 6. All of the interdiction formulations and algorithms were implemented in Julia v1.1 using the optimization modeling layer JuMP.jl v0.18 \cite{Dunning2017} and PowerModels v0.8 \cite{PowerModels}. Furthermore, an optimality tolerance of $\varepsilon = 1\%$ was used as a termination criteria. The code and the data used for the case studies are open-sourced and available in the repository \url{https://github.com/kaarthiksundar/resource-constrained-nk}. The computational experiments on the aforementioned transmission systems are aimed at corroborating the following aspects: (i) the effectiveness of the constraint-generation algorithm in computing optimal solutions to the spatially and topologically resource-constrained interdiction problems (ii) the choice of the dual bounds in Sec. \ref{subsec:dual-bounds} on the convergence behaviour of the algorithms, and finally (iii) the value of modeling a spatial- and topological-constrained adversary as opposed to an traditional adversary in the context of interdiction problems.

\subsection{Computational performance and choice of dual bounds}
Fig. \ref{fig:iter} shows the number of iterations the constraint-generation algorithm needs to converge to an $\varepsilon$-optimal solution on the RTS96 single-area system with valid and heuristic dual upper bounds for the spatial and topological $N$-$k$ interdiction problems. 
For this set of experiments on the RTS96 test system $D$ was set to $10$ km. From Fig. \ref{fig:iter}, it is clear that the computational performance of the algorithm is superior when the heuristic dual upper bounds are used. Despite the lack of a proof of validity for these bounds, they still yield optimal solutions to the interdiction problems. Hence, throughout the rest of the paper, heuristic dual upper bounds are used for all experiments. 
\begin{figure}[h!]
\centering
\begin{tikzpicture}[scale=0.8]
    \begin{axis}[
    xlabel=$k$, 
    ylabel=\text{itertions}, 
    legend style={at={(1,1.2)}, column sep=5pt, draw=none, legend columns=2, /tikz/column 2/.style={
                column sep=6pt,}},
    ymode=log]
    
    \addplot[only marks,color=red,mark options={scale=1.5,fill=white},mark=*] 
    plot coordinates {
        (2,320) (3,1740) (4,2034) (5,2335) (6,2950) 
    };
    \addplot[only marks,color=red,mark options={scale=1.5,opacity=0.5},mark=*] 
    plot coordinates {
        (2,20) (3,19) (4,14) (5,12) (6,14) 
    };
\addplot[only marks,color=blue,mark options={scale=1.5,fill=white}, mark=square*] 
plot coordinates {
        (2,77) (3,179) (4,454) (5,1089) (6,1205) 
    };
\addplot[only marks,color=blue,mark options={scale=1.5,opacity=0.5},mark=square*] 
plot coordinates {
        (2,11) (3,11) (4,19) (5,11) (6,35)
    };
    \legend{Spatial (V),Spatial (H),Topological (V), Topological (H)}
    
    \end{axis}
\end{tikzpicture}
\caption{Number of iterations required to converge to an $\varepsilon$-optimal solution with valid (V) and heuristic (H) dual upper bounds for the spatial and the topological $N$-$k$ interdiction problem.}
\label{fig:iter}
\end{figure}
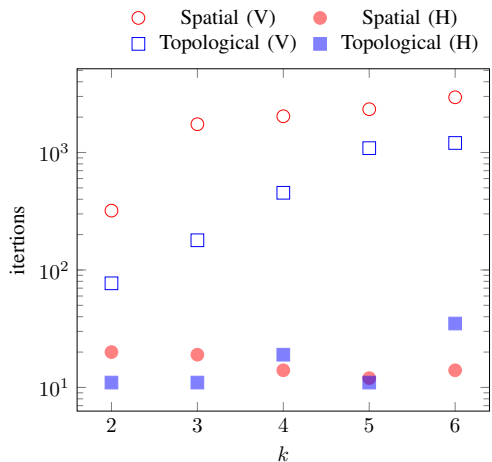

\begin{table}[htbp]
    \caption{Results comparing the different interdiction models on the IEEE single-area RTS96 test system. As $k$ gets larger, the effect of constraining the attacker grows considerably. The load shed values obtained using a traditional approach for determining the interdiction can be as much as 50\% more than a spatially constrained attacker.}
    \label{tab:96}
    \centering
    \begin{tabular}{crrrr}
        \toprule
        \multirow{2}{*}{$k$} & load shed & time  & \multirow{2}{*}{iterations} & opt. gap  \\
        & (p.u) & (sec.) & & (\%) \\
        \midrule
        \multicolumn{5}{l}{Traditional $N$-$k$ interdiction} \\[0.5ex]
        2 & 4.0 & 4.42 & 21 & 0.00 \\
        3 & 7.37 & 3.45 & 15 & 0.50 \\
        4 & 11.05 & 3.43 & 11 & 0.33 \\
        5 & 14.21 & 2.96 & 10 & 0.86 \\
        6 & 15.96 & 3.43 & 13 & 0.00 \\
        \midrule
        \multicolumn{5}{l}{Spatial $N$-$k$ interdiction} \\[0.5ex]
        2 & 4.0 & 2.76 & 20 & 0.00 \\
        3 & 6.3 & 2.56 & 20 & 0.00 \\
        4 & 8.0 & 2.44 & 19 & 0.70 \\
        5 & 9.59 & 2.97 & 25 & 0.00 \\
        6 & 11.0 & 2.38 & 19 & 0.00 \\
        \midrule
        \multicolumn{5}{l}{Topological $N$-$k$ interdiction} \\[0.5ex]
        2 & 4.0 & 4.39 & 11 & 0.00 \\
        3 & 6.29 & 4.10 & 11 & 0.24 \\
        4 & 7.72 & 5.31 & 19 & 0.00 \\
        5 & 11.05 & 4.85 & 11 & 0.00 \\
        6 & 11.05 & 20.09 & 35 & 0.00 \\
         \bottomrule
    \end{tabular}
\end{table}

Table \ref{tab:96} shows the $\varepsilon$-optimal load shed values, computation time, number of iterations, and the relative optimality gap for the traditional, spatial (with $D=10$ km) and topological interdiction problems for varying values of $k$. Table \ref{tab:240} shows the same for the WECC 240 test system. These tables corroborates the effectiveness of the algorithm in computing $\varepsilon$-optimal solutions with very little computational effort.  

\subsection{Value of realistic attacker models}
From Tables \ref{tab:96} and \ref{tab:240} we observe that the load shed obtained using the traditional $N$-$k$ interdiction formulation is much higher than those obtained from the spatial and the topological interdiction problems. This is because the traditional $N$-$k$ permits the attacker to interdict components in the system that are geographically far apart in the transmission system. This, however, is less realistic and, as shown by the results, can significantly overestimate the required load shed values.
\begin{table}[htbp]
    \caption{
    Results comparing the different interdiction models on the WECC 240 test system. As $k$ gets larger, the effect of constraining the attacker grows considerably. The traditional approach for determining the interdiction can be almost 200\% more than a constrained attacker. Interestingly, constraining the attacker also has the side effect of reducing the number of iterations that are required for convergence.}
    \label{tab:240}
    \centering
    \begin{tabular}{crrrr}
        \toprule
        \multirow{2}{*}{$k$} & load shed & time  & \multirow{2}{*}{iterations} & opt. gap  \\
        & (p.u) & (sec.) & & (\%) \\
        \midrule
         \multicolumn{5}{l}{Traditional $N$-$k$ interdiction} \\[0.5ex]
        2 & 219.19 & 3.71 & 14 & 0.00 \\
        3 & 331.8 & 4.9 & 19 & 0.00 \\
        4 & 418.89 & 4.64 & 16 & 0.06 \\
        5 & 482.22 & 7.48 & 24 & 0.80 \\
        6 & 556.65 & 5.63 & 18 & 0.77 \\
        \midrule
        \multicolumn{5}{l}{Spatial $N$-$k$ interdiction} \\[0.5ex]
        2 & 192.22 & 3.03 & 5 & 0.00 \\
        3 & 222.65 & 3.59 & 7 & 0.00 \\
        4 & 233.99 & 5.22 & 11 & 0.57 \\
        5 & 255.73 & 4.98 & 10 & 0.94 \\
        6 & 273.68 & 7.15 & 16 & 0.0 \\
        \midrule
        \multicolumn{5}{l}{Topological $N$-$k$ interdiction} \\[0.5ex]
        2 & 121.26 & 5.52 & 5 & 0.40 \\
        3 & 211.26 & 5.52 & 4 & 0.00 \\
        4 & 222.49 & 9.72 & 5 & 0.88 \\
        5 & 233.4 & 41.64 & 12 & 0.00 \\
        6 & 332.03 & 21.27 & 6 & 0.00 \\
         \bottomrule
    \end{tabular}
\end{table}

This phenomenon is demonstrated in the spatial interdiction model by examining the impact of parameter $D$ on the load shed values. To that end, we let $D$ take any value from $100$km to $1000$km in steps of $100$km. Fig. \ref{fig:load-shed} shows the load shed values for various values of $D$ and $k$. The load shed for fixed values of $k$ quickly increases with $D$ (and will converge to the load-shed value for the traditional $N-k$ formulation when $D$ is sufficiently large), thus establishing the importance of a more accurate attacker model. Fig. \ref{fig:iterations} shows the number of iterations needed by the constraint-generation algorithm to converge to an $\varepsilon$-optimal solution for the spatial $N$-$k$ interdiction problem. The plot suggests that the difficulty of the problem does not increase or decrease substantially with the value of the parameter $D$ and for a fixed value of $k$. 
\begin{figure}
    \centering
    \includegraphics[scale=0.7]{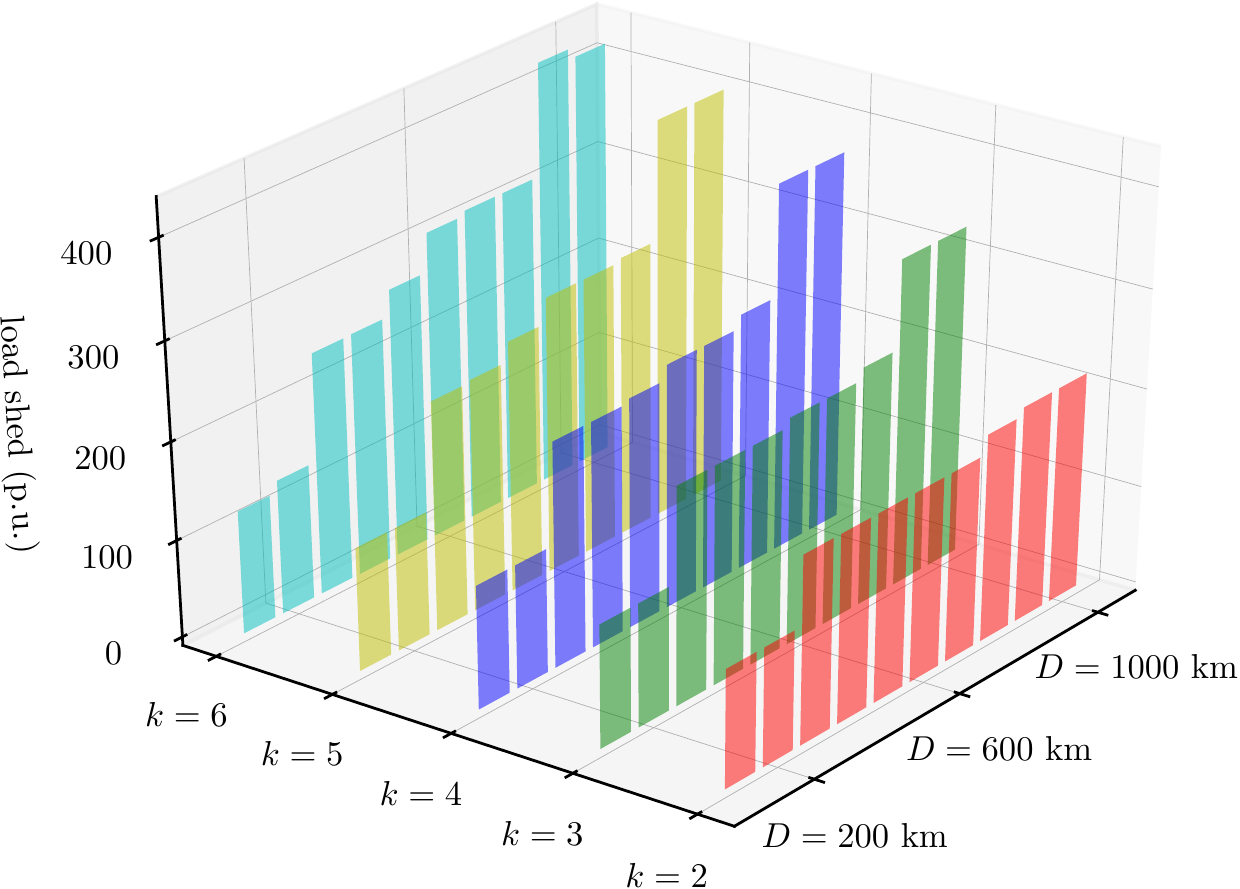}
    \caption{Load shed, in p.u., for the spatial interdiction problem for varying values of $k$ and $D$ on the WECC 240 test system.}
    \label{fig:load-shed}
\end{figure}

\begin{figure}
    \centering
    \includegraphics[scale=0.7]{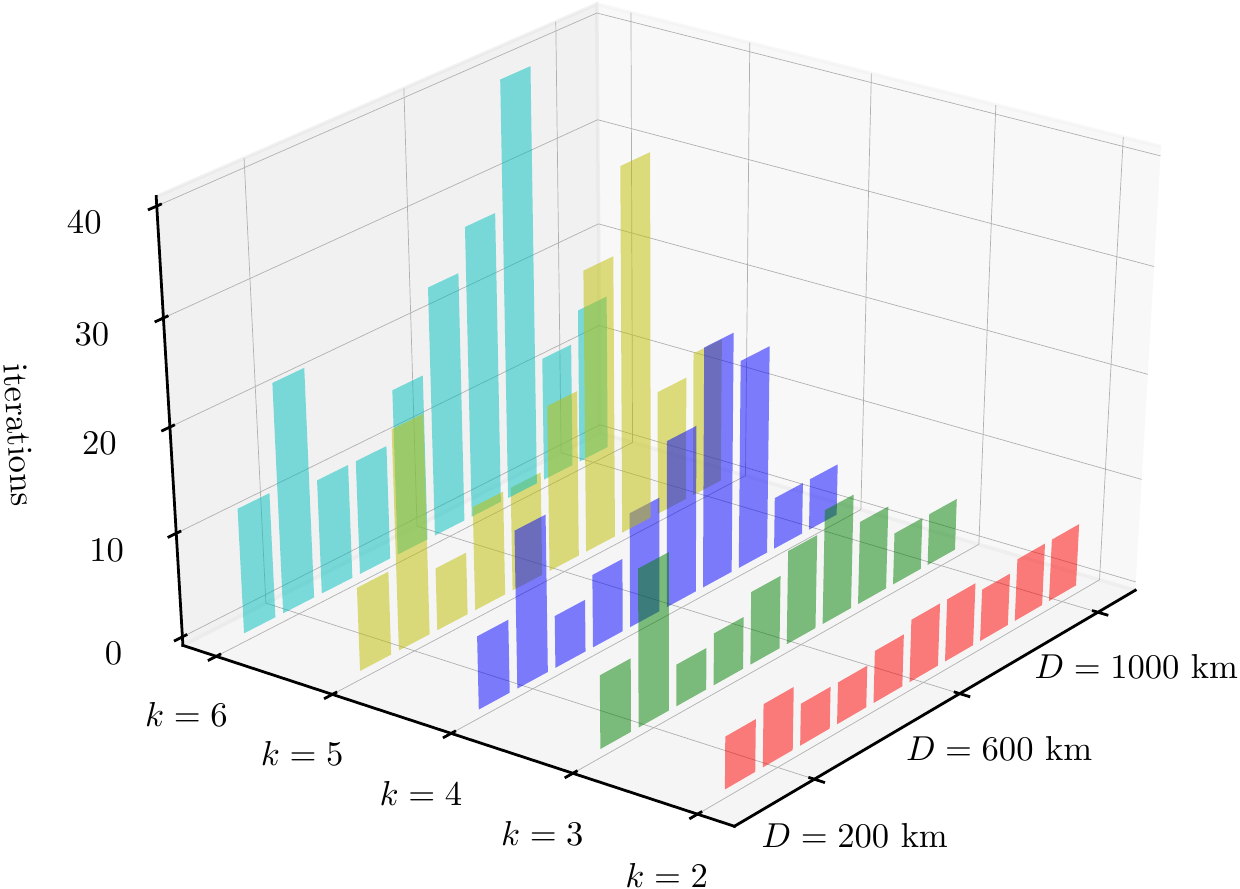}
    \caption{Number of iterations of the constraint-generation algorithm for the spatial interdiction problem for varying values of $k$ and $D$ on the WECC 240 test system.}
    \label{fig:iterations}
\end{figure}

The final set of plots in Fig. \ref{fig:maps} shows locations where the load is shed for the WECC 240 test system by the traditional  (see Fig. \ref{fig:traditional}), spatial with $D = 500$ km (see Fig. \ref{fig:spatial}), and topological (see Fig. \ref{fig:top}) interdiction models, respectively, for $k=6$. As observed from Fig. \ref{fig:maps}, the traditional interdiction problem has the ability to compute $N$-$k$ attacks where the components interdicted have a large geographical separation. As a result, the attack produces a load shed in geographically different parts of the system. Both the spatial and the topological $N$-$k$ models result in $N$-$k$ attacks where the load shed is more localized to a particular region indicative of the fact that the topological $N$-$k$ model serves as a proxy for the spatial interdiction problem in the absence of spatial grid data. 

\textcolor{black}{
\subsection{Comment on computation time for large-scale systems}
We first remark that the spatial $N$-$k$ interdiction problem with a large value of $D$ implies that the attacker is getting more powerful and can start interdicting components in the network which are geographically well-separated from each other. This is in opposition to the paper's motivation on computing the worst-case credible $N$-$k$ interdiction plan. Hence, in this scenario, we would advise the readers to resort to the traditional $N$-$k$ interdiction problem which does not face any computational issues for large-scale systems. Nevertheless, if we assume that a credible attacker can attack a system with a large $D$ value the issue that the algorithm would run into is redundancy in the solution for systems larger than the ones considered in this paper. For a sufficiently large $D$ and small $k$ value, there can be multiple interdiction locations for the same $N$-$k$ interdiction plan causing a sufficient increase in the computation time for convergence of the cutting-plane algorithm. The cutting-plane approach as presented in Algorithm \ref{algo:pseudocode} can easily tackle this redundancy using the no-good cut in Eq. \eqref{eq:no-good}, that is added at each iteration of the algorithm. This constraint forces any new $N$-$k$ interdiction plan to differ from interdiction plan chosen in a previous iteration by at least one component. }

\begin{figure*}
 \centering
  \subfloat[Load shed = 556.65 p.u.\label{fig:traditional}]{%
       \includegraphics[scale=0.15,cfbox=black 1pt 1pt]{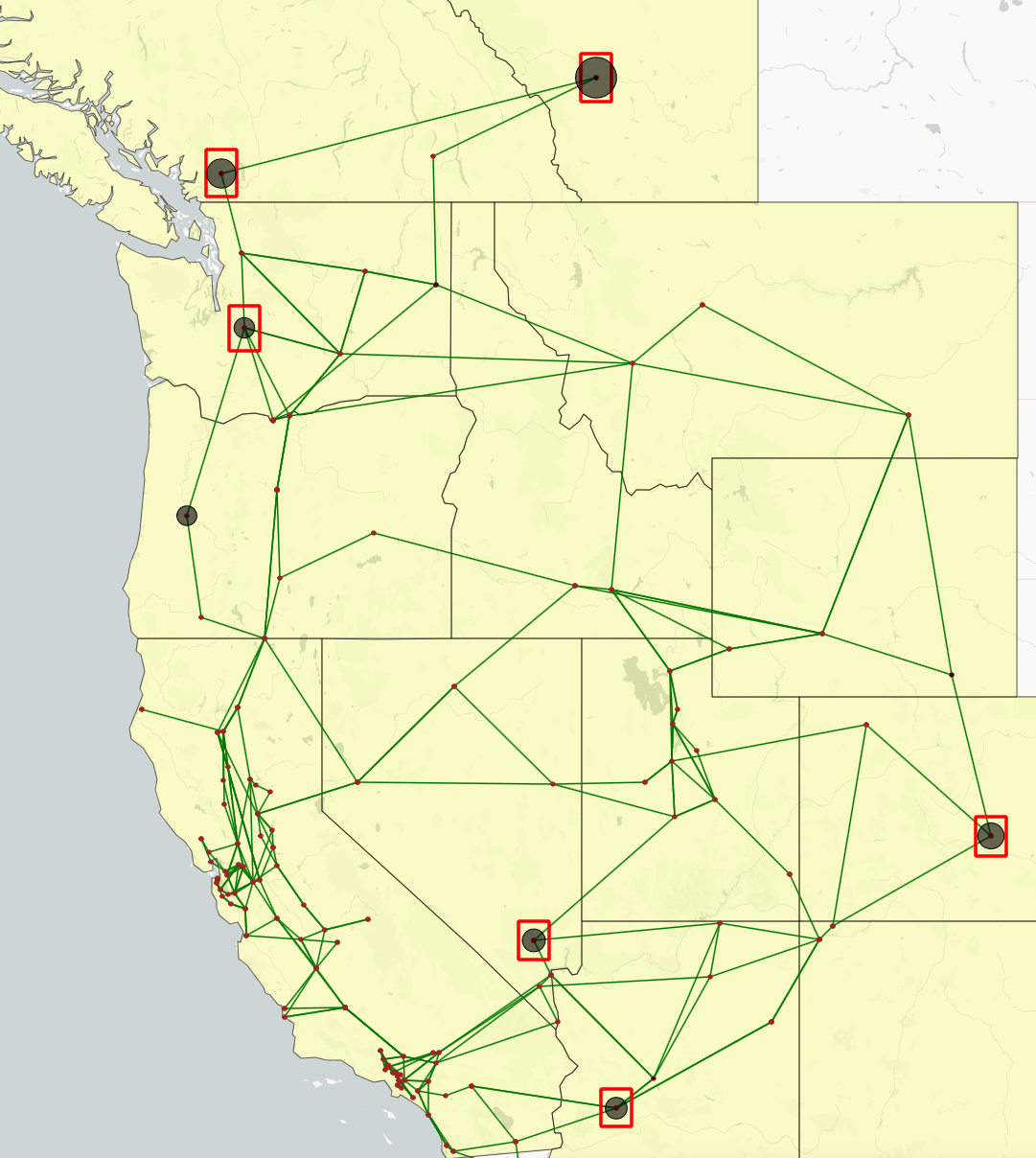}}
    \hfill
  \subfloat[Load shed = 273.68 p.u\label{fig:spatial}]{%
        \includegraphics[scale=0.15,cfbox=black 1pt 1pt]{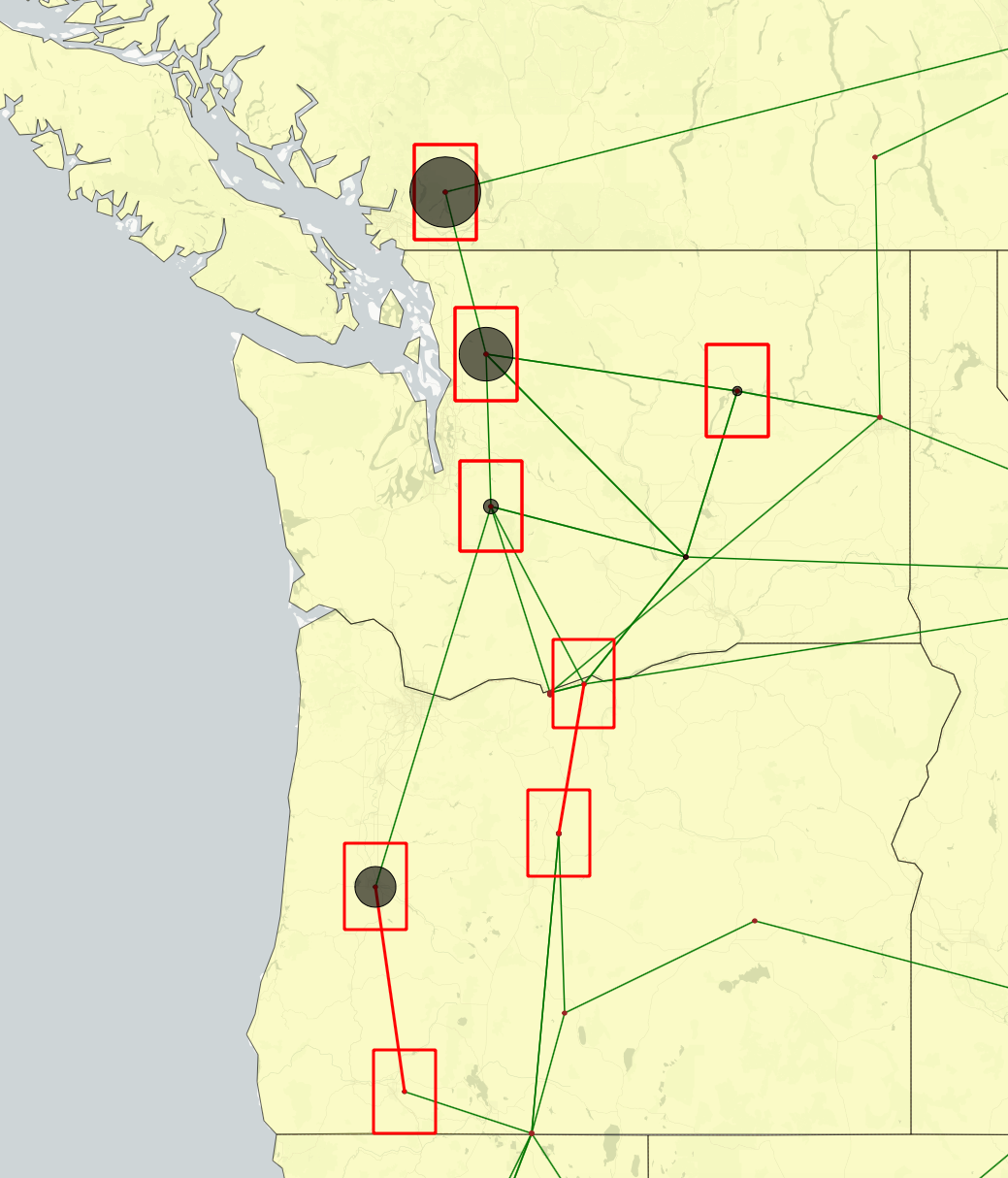}}
    \hfill
  \subfloat[Load shed = 332.03 p.u.\label{fig:top}]{%
        \includegraphics[scale=0.15,cfbox=black 1pt 1pt]{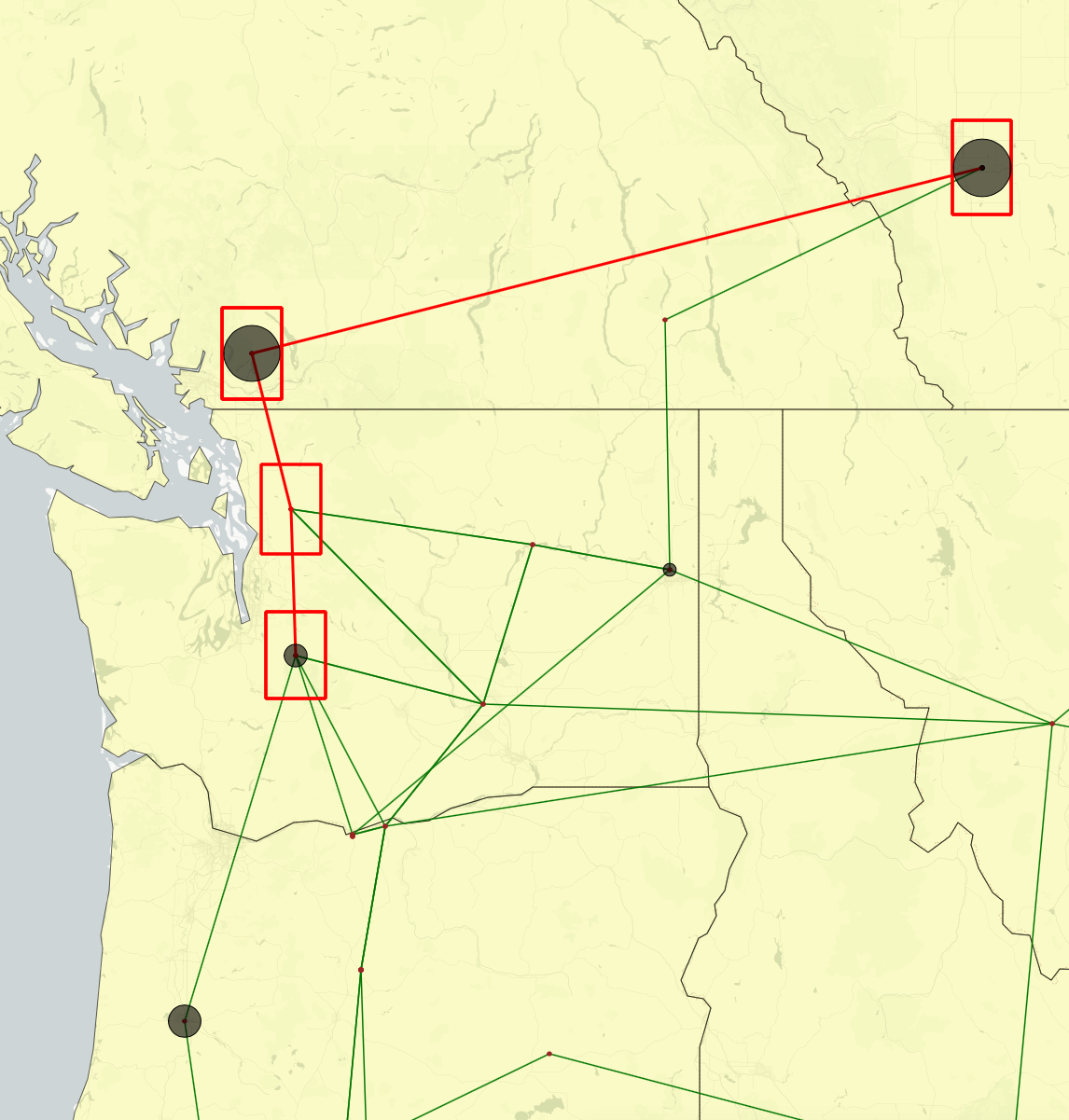}}
  \caption{Map of the buses where the load is shed for the (a) traditional, (b) spatial, and (c) topological interdiction problems with $k=6$. The radius of the each of the circles centered around a bus represents the percentage of total load shed in that bus. The interdicted lines are shown in red. All the buses on which the interdicted lines are incident on are highlighted using red boxes. We also remark that some of the interdicted lines in all the three figures connect buses that are very close to each other and hence is not highlighted in red, and hence the reason for highlighting the buses on which the lines are incident on. For the spatial interdiction problem, the value of $D$ was set to $500$ km.}
  \label{fig:maps} 
\end{figure*}

\section{Conclusion}
This paper presents two $N$-$k$ interdiction models 
that constrain the capabilities of the attacker to better model localized interdiction. These models use spatial proximity and topological connectivity for formulate localized interdiction.
A general constraint-generation algorithm was developed to handle a broad class of interdiction problems and was demonstrated these localized interdiction formulations and the traditional interdiction models. Case studies on the IEEE RTS 96 and WECC 240 systems show the computational effectiveness of the algorithm and the effectiveness of the model in identifying attacks that are localized. 

There are a number of potentially interesting future directions for this work. First, new models of constraining the attack model to better reflect the capabilities and outcomes of extreme events, such as those based on stochastic $N$-$k$ models could be developed \cite{Sundar2018}. Second, these localized models could be connected to AC power flow models, such as the convex relaxations developed in  \cite{hijazi2017convex}, to improve the realism of the chosen attacks. 
Third, interdiction models typically use steady-state models that neglect transient responses to failures in power systems. It would be interesting to develop computationally tractable models of interdiction that model the dynamics of power systems. Fourth, the interdiction literature models attackers with complete information about the transmission system and the defender's response and developing models where the attacker has partial or incomplete information of the transmission system and the defender's response is another avenue of future work.
Finally, it would also be interesting to develop methods for deriving stronger valid bounds on the dual variables.

\section*{Acknowledgements} 
The  work  was  funded  by  the U.S. Department of Energy (DOE) Grid Modernization Laboratory Consortium project \textit{Extreme Event Modeling}. Los Alamos National Laboratory is operated by Triad National Security, LLC, for the National Nuclear Security Administration of DOE (Contract No. 89233218CNA000001) under the auspices of the NNSA of the U.S. DOE at LANL under Contract No. DE-AC52-06NA25396.

\bibliographystyle{plain}
\bibliography{references.bib}

\end{document}